\documentclass[letterpaper,10pt,conference]{ieeeconf}
\IEEEoverridecommandlockouts              

\usepackage{epsfig} 
\usepackage{psfrag}
\usepackage{subfigure}
\usepackage{tikz}
\makeatletter
\usepackage{amsmath} 
\usepackage{amssymb} 
\usepackage{color}

\usepackage{amsthm}

\DeclareMathOperator{\Tr}{Tr}
\DeclareMathOperator{\Deg}{deg}

\newtheorem{theorem}{Theorem}
\newtheorem{remark}{Remark}
\newtheorem{assumption}{Assumption}

\title{\LARGE \bf
Domain Decomposition for Stochastic Optimal Control
}

\author{Matanya B. Horowitz \quad Ivan Papusha \quad Joel W. Burdick
\thanks{I. Papusha was supported by a Department of Defense NDSEG Fellowship.
The authors are with the Control and Dynamical Systems Department, California
Institute of Technology, Pasadena, CA USA. The corresponding author may be
reached at {\tt\small mhorowit@caltech.edu}
}}%

\begin{document}

\maketitle
\thispagestyle{empty}
\pagestyle{empty}

\begin{abstract}
This work proposes a method for solving linear stochastic optimal
control (SOC) problems using sum of squares and semidefinite programming.
Previous work had used polynomial optimization to approximate the value
function, requiring a high polynomial degree to capture local phenomena. To
improve the scalability of the method to problems of interest, a domain
decomposition scheme is presented. By using local approximations, lower degree
polynomials become sufficient, and both local and global properties of the
value function are captured. The domain of the problem is split into a
non-overlapping partition, with added constraints ensuring $C^1$ continuity.
The Alternating Direction Method of Multipliers (ADMM) is used to optimize over
each domain in parallel and ensure convergence on the boundaries of the
partitions. This results in improved conditioning of the problem and allows for
much larger and more complex problems to be addressed with improved
performance. 
\end{abstract}

\section{Introduction}

Motion planning in the presence of noise and dynamics remains a central issue
in robotics and autonomous systems. As robots transition out of controlled
factory and lab environments, the ability to move precisely in the presence of
unknown environments, exterior agency, and stochastic actuators and sensors
become ever more important. For a solution to be useful it must be rapid to
compute, robust, and incorporate optimality criteria. The primary avenue for
solving motion planning problems, and likely the most successful historically,
has been that of sampling based planners~\cite{LaValle:tb}. Such approaches are
attractive as they may be quite rapid in practice, but typically only have
guarantees in the asymptotic limit, and incorporate dynamics and stochasticity
in only a limited way. 

Stochastic optimal control (SOC) provides an alternative, allowing for the full
dynamics and various details of the problem to be incorporated into the
algorithm directly. Traditionally, this has been handled through
discretization, resulting in the formulation of a Markov Decision Problem
(MDP), which can then be solved through methods such as value
iteration~\cite{Bertsekas:2005}.  These methods have met with a great deal of
success in a number of communities.  The caveat is that such problems in
robotics may be prohibitively difficult to solve due to a number of obstacles,
chiefly the \emph{curse of dimensionality}. These techniques rely on a fine
discretization of the state space when the system occupies a continuous domain,
typical of many robotic and control problems. Furthermore, robotic state spaces
are usually quite large, both in quantity of dimensions as well as absolute
size, for all but the most academic of problems, resulting in discrete state
space cardinality that may easily exceed the capabilities of current computers.
Reducing the necessity for fine discretization could provide for significant
gains in this area. 

Recently it has been discovered that the Hamilton Jacobi Bellman (HJB)
equation, a typically nonlinear partial differential equation (PDE) that arises
in optimal control, may be transformed to a linear PDE given several mild
assumptions. This is a large computational gain, as solving the nonlinear PDE
is quite difficult~\cite{Fleming:2006tl}. Research into leveraging this
computational advantage is only beginning.

One method to solve such problems lies in recent results from polynomial
optimization and semidefinite programming~\cite{Parrilo:2000ui}. These methods
allow for optimization to be performed directly over polynomials, and have
solved a number of difficult problems. Here we present a novel use of such
tools to directly construct an approximate value function that satisfies the
linear HJB equation. This allows for optimal control problems, including those
typically found in robotic motion planning, to be solved relatively quickly and
globally. In contrast to dynamic programming approaches, no direct state space
discretization is required, postponing the curse of dimensionality and
eliminating a potential source of approximation error.

In particular, we propose an augmentation of the algorithm first presented
in~\cite{Horowitz:2014tu}, in which the domain is split into distinct
partitions, each of which has its own local approximating polynomial. The value
function may vary significantly over the domain, and thus may require a high
degree polynomial if approximated over the domain's entirety. But by using a
sufficiently local approximation, a similar quality of global approximation may
be achieved with smaller degree on each partition. Furthermore, we demonstrate
that an efficient choice of partitioning may lead to a decoupling in the
optimal control problems on each partition, allowing for a degree of
parallelization. The Alternating Direction Method of Multipliers
(ADMM)~\cite{Boyd:2011bw} is a particularly well suited approach, providing a
principled method for parallelization of certain convex problems with
convergence guarantees.

\subsection{Related Work}

Linearly solvable SOC problems have recently been studied from two avenues. One
is Linear MDPs~\cite{Todorov:2009wja}, in which an MDP may be solved as a
linear set of equations given several assumptions. By taking the continuous
limit of the discretization, a linear PDE is obtained. Additionally, following
the work begun by Kappen~\cite{Kappen:2005kb}, the same linear PDE has been
found through a particular transformation of the HJB. The existing research has
tended towards developing sampling based approaches for solving the resulting
linear PDE. This is done through the use of the Feynman--Kac Lemma, that allows
for a linear PDE to be solved by examining the diffusion of a stochastic
process. Feynman--Kac approaches have been further developed by Theodorou et
al.~\cite{Theodorou:2010vd} into a path integral framework in use with
dynamic motion primitives.  These results have grown in a number of
compelling directions, either relying on an MDP or sampling based
approach~\cite{Theodorou:2012wx,Theodorou:2010vd,Dvijotham:2011tv}. 

Sampling based approaches are an alternative to the approach presented
here, with several potential advantages and disadvantages.  Among these,
sampling based approaches such as that of Theodorou et al. may be more amenable
in high dimensional state spaces. Such a comparison in part motivates the
present work.

Effort has also gone towards solving the linear HJB directly, as well as
exploiting its properties for computational benefit. In~\cite{Horowitz:2014we}
it is shown that the property of superposition may be used to compute optimal
control solutions at essentially zero computational cost, with significant
implications in solving Linear Temporal Logic (LTL) specified tasks. The work
of~\cite{Horowitz:hd} leverages recent results in sparse tensor decompositions
to formulate a numerical technique that scales \emph{linearly} with dimension,
allowing for the HJB to be approximately solved for a twelve dimensional
system. Finally, in~\cite{horowitz2014efficient} connections are made to a
broader literature, such as \emph{navigation} functions (popular in robotics),
problems of moments, and broader classes of linear PDEs.

The sum of squares approach presented here is connected via duality to problems
of moments. By examining the moments of the HJB, an alternative line of work by
Lasserre et al.~\cite{lasserre2009moments, lasserre2008nonlinear,
lasserre2001global} also reduces optimal control to a semidefinite optimization
problem. In their work, the solution and the optimality conditions are
integrated against monomial test functions, producing an infinite set of moment
constraints. By truncating to any finite list of monomials, the optimal control
problem is reduced to one of semidefinite optimization. Their method is more
general, applicable to any system with polynomial nonlinearities. Our method
contrasts in that we propose candidate solutions of the value function, and
thus avoid the need to include the control signal in our polynomial basis,
lessening the computational burden. We are also able to avoid consideration of
initial and final conditions and measures. Perhaps most importantly, our use of
the linear HJB allows for both upper and lower pointwise bounds to be
constructed to the true solution.

Domain partitioning is an approach that has long been used in numerical
methods for PDEs, from the local analysis behind the Finite Element Method to
multi-scale decomposition techniques~\cite{bjorstad2004domain}.  In control,
these techniques have also arisen to improve local approximation to Lyapunov
functions~\cite{johansson1998computation}, and is complimentary to approaches
that approximate nonlinear systems as piecewise-affine
(PWA)~\cite{biswas2005survey}. Our work may be seen as a continuation of this
research theme, seeking to extend these techniques not only to the study of
stability, as is the case for Lyapunov functions, but to control as well.
Furthermore, the ability to obtain general solutions to HJB has implications
in regards to Control Lyapunov Functions~\cite{Sontag:1983cq}, allowing for
stabilization to be shown alongside near-optimality. Our method has the
distinct advantage over PWA approximations in that the system itself is not
approximated, and the full nonlinear dynamics are incorporated into the
solution.

\subsection{Paper Outline}
The ability to perform domain decomposition for stochastic optimal control will
rely on three main ideas: linear stochastic optimal control, sum of squares
programming, and finally the ADMM algorithm. We begin in Section
\ref{sec:The_Linear_HJB} by reviewing the linear HJB. In Section
\ref{sec:sos_relaxation}, we develop the technique first presented
in~\cite{Horowitz:2014tu} to approximately solve the linear HJB using convex
programming via a sum of squares relaxation. Finally, we build the domain
decomposition procedure in Section \ref{sec:domain_decomposition}. The need to
enforce constraints on the boundaries between partitions then gives rise to our
use of ADMM, which is reviewed in Section \ref{sec:admm} and then applied to
the problem at hand, the main contribution of this paper. We illustrate each
step on a simple nonlinear example in Section \ref{sec:illustrative_example},
before tackling a more sophisticated example in Section
\ref{sec:nonlinear_robot}. Finally, we discuss some of the merits of the
technique and future directions in Section \ref{sec:discussion}.

\section{The Linear Hamilton Jacobi Bellman Equation}
\label{sec:The_Linear_HJB}

We begin by constructing the value function, which captures the
\emph{cost-to-go} from a given state. If such a quantity is known, an optimal
action is chosen to follow the quantity's gradient, bringing the agent into
states with lowest cost over the remaining time horizon.  We define
$x_{t}\in\mathbb{R}^{n}$ as the system state at time $t$, control input
$u_{t}\in\mathbb{R}^{m}$, and dynamics that evolve according to the equation
\begin{equation}
	dx_{t} = 
	\left(f\left(x_{t}\right)+G\left(x_{t}\right)u_{t}\right)dt + 
	B\left(x_{t}\right)d\omega_{t}
	\label{eq:dynamics}
\end{equation}
on a compact domain $\Omega$, where the expressions $f(x)$, $G(x)$, $B(x)$ are
assumed to be smoothly differentiable, but possibly nonlinear functions, and
$\omega_{t}$ is a zero mean Gaussian noise process with covariance
$\Sigma_\epsilon$.  The system has cost $r_{t}$ accrued at time $t$ according
to
\begin{equation}
	r\left(x_{t},u_{t}\right) = 
	q\left(x_{t}\right)+\frac{1}{2}u_{t}^{T}Ru_{t}
	\label{eq:cost}
\end{equation}
where $q(x)$ is a state dependent cost. We require $q(x)\ge0$ for all $x$ in
the problem domain. The goal is to minimize the expectation of the cost
functional 
\begin{equation}
	J(x,u) = \phi_{T}\left(x_{T}\right) + 
	\int_{0}^{T}r\left(x_{t},u_{t}\right)dt,
	\label{eq:trajectory-cost}
\end{equation}
where $\phi_{T}$ represents a state-dependent terminal cost. The solution to
this minimization is obtained from the \emph{value function}. For an
initial point $x_{0}$, it is given by
\begin{equation}
	V\left(x_{0}\right)=
	\min_{u_{[0,T]}}\mathbb{E}\left[J\left(x_{0}\right)\right],
	\label{eq:value-def}
\end{equation}
where we use the shorthand $u_{[0,T]}$ to denote the trajectory of $u(t)$ over
the time interval $t\in[0,T]$.

The associated Hamilton Jacobi Bellman equation, arising from dynamic
programming arguments~\cite{Fleming:2006tl}, is
\begin{equation}
	-\partial_{t}V =
	\min_{u}\Big(r+\left(\nabla_{x}V\right)^{T}f + 
	\frac{1}{2}\Tr\left(\left(\nabla_{xx}V\right)G\Sigma_{\epsilon}G^{T}\right)\Big).
	\label{eq:HJB-raw}
\end{equation}
As the control effort enters quadratically into the cost function, it is a
simple matter to solve for it analytically by substituting~\eqref{eq:cost} into
\eqref{eq:HJB-raw} and taking the gradient, yielding
\begin{equation}
	u^{*}=-R^{-1}G^{T}\left(\nabla_{x}V\right).
	\label{eq:optimal-u}
\end{equation}

The optimal control $u^{*}$ may then be substituted into~\eqref{eq:HJB-raw} to
yield the following nonlinear, second order PDE
\begin{multline}
	\label{eq:HJB-nonlinear}
	-\partial_{t}V = q +\left(\nabla_{x}V\right)^{T}f
	- \frac{1}{2}\left(\nabla_{x}V\right)^{T}GR^{-1}G^{T}\left(\nabla_{x}V\right) \\
	+ \frac{1}{2} \Tr\left(\left(\nabla_{xx}V\right)B\Sigma_{\epsilon}B^{T}\right).
\end{multline}

The difficulty of solving this PDE has traditionally prevented the value
function from being solved for directly. However, as has recently been found
in~\cite{Kappen:2005bn,Todorov:2009wja}, if there exists a scalar
$\lambda > 0$ and a control penalty cost $R\in\mathbb{R}^{n\times n}$
satisfying the noise assumption
\begin{equation}
	\lambda G(x)R^{-1}G(x)^{T}=
	B(x)\Sigma_{\epsilon}B(x)^{T}\triangleq\Sigma_{t},
	\label{eq:noise-assumption}
\end{equation}
then the logarithmic transformation
\begin{equation}
	V=-\lambda\log\Psi\label{eq:log-transform}
\end{equation}
allows us to obtain, after substitution and simplification, the following
linear PDE from equation~\eqref{eq:HJB-nonlinear},
\begin{equation}
	-\partial_{t}\Psi= 
	-\frac{1}{\lambda}q\Psi+f^{T}\left(\nabla_{x}\Psi\right)
	+\frac{1}{2}\Tr\left(\left(\nabla_{xx}\Psi\right)\Sigma_{t}\right).
	\label{eq:HJB-linear}
\end{equation}
Through the transformation $\Psi$, which we call here the
\emph{desirability}~\cite{Todorov:2009wja}, we obtain a computationally
appealing method from which to compute the value function $V$.

\begin{remark}
The noise assumption \eqref{eq:noise-assumption} can roughly be interpreted as
a controllability-type condition: the system controls must span (or
counterbalance) the effects of input noise on the system dynamics. A degree of
designer input is also given up, as the constraint restricts the design of the
control penalty $R$, requiring that control effort be highly penalized in
subspaces with little noise, and lightly penalized in those with high noise.
Additional discussion may be found in~\cite{Todorov:2009wja}.
\end{remark}

The boundary conditions of \eqref{eq:HJB-linear} correspond to the
exit conditions of the optimal control problem. This may correspond
to colliding with an obstacle or goal region, and in the finite horizon
problem there is the added boundary condition of the terminal cost
at $t=T$. These final costs must then be transformed according to
\eqref{eq:log-transform}, producing added boundary conditions to
\eqref{eq:HJB-linear}.

Linearly solvable optimal control is not limited to the finite horizon
setting. Similar analysis can be performed to obtain linear HJB PDEs
for infinite horizon average cost, and first-exit settings, with the
corresponding cost functionals and PDEs shown in Table \ref{tab:control_pdes}.
For convenience, we define the differential operator
\begin{equation}
\label{eq:L}
L(\Psi):=f^{T}\left(\nabla_{x}\Psi\right)+\frac{1}{2}\Tr\left(\left(\nabla_{xx}\Psi\right)\Sigma_{t}\right).
\end{equation}

\begin{table}
\vspace{3mm}
\caption{Linear Desirability PDE for Various Stochastic
Optimal Control Settings, from~\cite{Todorov:2009wja}.}
\label{tab:control_pdes}
\begin{centering}
\begin{tabular}{|c|c|c|}
\hline 
 & Cost Functional & Desirability PDE \\
\hline 
Finite & $\phi_{T}(x_{T})+\int_{0}^{T}r(x_{t},u_{t})dt$ & $\frac{1}{\lambda}q\Psi-\frac{\partial\Psi}{\partial t}=L(\Psi)$\\
\hline 
First-Exit & $\phi_{T_{*}}(x_{T_{*}})+\int_{0}^{T}r(x_{t},u_{t})dt$ & $\frac{1}{\lambda}q\Psi=L(\Psi)$\\
\hline 
Average & $\lim_{T\to\infty}\frac{1}{T}\mathbb{E}\left[\int_{0}^{T}r(x_{t},u_{t})dt\right]$ & $\frac{1}{\lambda}q\Psi-c\Psi=L(\Psi)$\\
\hline
\end{tabular}
\end{centering}
\end{table}

\section{The Sum of Squares Relaxation}
\label{sec:sos_relaxation}
Building upon the results of~\cite{Horowitz:2014tu}, we relax the equality
constraint \eqref{eq:HJB-linear}, allowing for an over-approximation of the
value function, and creating a linear differential inequality. This places the
problem within the realm of polynomial optimization problems where tools such
as the \emph{Positivstellensatz} may be applied. Consider the relaxation
\begin{equation}
\label{eq:relax_hjb}
	\frac{1}{\lambda}q\Psi
	\ge
	\partial_{t}\Psi
	+ f^{T}(\nabla_{x}\Psi)
	+ \frac{1}{2}\Tr\left(\left(\nabla_{xx}\Psi \right)\Sigma_{t}\right).
\end{equation}
Given that this is an approximation, we wish to obtain the best such
approximation for a given polynomial order for $\Psi$, minimizing the pointwise
error as our objective,
\[
	\begin{array}{rl}
		\mbox{min.} & \gamma \\
		\mbox{s.t.} & \displaystyle \gamma - 
			\left(\frac{1}{\lambda} q \Psi -
				\left(\partial_t \Psi + L\left(\Psi \right) \right) \right) \ge 0.
	\end{array}
\]
Furthermore, due to the nature of the log
transformation~\eqref{eq:log-transform}, we require $\Psi$ to be positive
everywhere, and we will examine this problem only on a compact, semialgebraic
domain
$\mathbb{S}$.  

The complete (centralized) optimization problem is
\begin{alignat}{2}
\mbox{min.} \quad & \gamma & \label{eq:sos_pde} \\
\mbox{s.t.} \quad & \frac{1}{\lambda}q\Psi \ge \partial_{t}\Psi+L(\Psi) & x\in\mathbb{S} \nonumber\\
 & \gamma \ge\frac{1}{\lambda}q \Psi- \partial_t \Psi -L(\Psi) & x\in\mathbb{S}  \nonumber\\
 & \Psi \ge e^{-\frac{\phi_T(x)}{\lambda}} & x\in\partial\mathbb{S} \nonumber\\
 & \gamma \ge \Psi - e^{-\frac{\phi_T(x)}{\lambda}} & x\in\partial\mathbb{S} \nonumber
\end{alignat}
The inequalities are interpreted pointwise over $x\in\mathbb{S}$. This set of
polynomial inequalities motivates our need for a method to enforce
non-negativity constraints over a polynomial directly.

\subsection{Sum of Squares Review}
We provide a brief review of sum of squares (SOS) programming, with additional
technical details available in~\cite{Parrilo:2003fh,lasserre2001global}. These
tools will be key in the development of approximate solutions to
\eqref{eq:sos_pde}. 

Formally, a \emph{semialgebraic set} is a subset of $\mathbb{R}^{n}$ that is
specified by a finite number of polynomial equations and inequalities.  An
example is the set
\[ 
	\left\{ \left(x_{1},x_{2}\right)\in\mathbb{R}^{2}\mid x_{1}^{2} 
	+ x_{2}^{2}\le1,x_{1}^{3}-x_{2}\le0\right\}. 
\]
Such a set is not necessarily convex, and testing membership in the set is
intractable in general~\cite{Parrilo:2003fh}. As we will see, however, there
exists a class of semialgebraic sets that are in fact
semidefinite-representable. Key to this development is the ability to test for
non-negativity of a polynomial.

A multivariate polynomial $f(x)$ is a \emph{sum of squares} (SOS) if there
exist polynomials $f_{1}(x),\ldots,f_{m}(x)$ such that
\[ 
	f(x)=\sum_{i=1}^{m}f_{i}^{2}(x).
\]
A seemingly unremarkable observation is that a sum of squares is always
positive. Thus, a sufficient condition for non-negativity of a polynomial is
that the polynomial is SOS. Perhaps less obvious is that membership in the set
of SOS polynomials may be tested as a convex problem. We denote the function
$f(x)$ being SOS as $f(x)\in\Sigma(x)$.

\begin{theorem} 
\label{thm:sos-test-sdp}
(\cite{Parrilo:2003fh}) Given a finite set of polynomials 
$\left\{ f_{i}\right\} _{i=0}^{m}\in\mathbb{R}[x]$ the existence of 
$\left\{ a_{i}\right\} _{i=1}^{m}\in\mathbb{R}$ such that
\[
	f_{0}+\sum_{i=1}^{m}a_{i}f_{i}\in\Sigma(x) 
\]
is a semidefinite programming feasibility problem. 
\end{theorem}
Here, $\mathbb{R}[x]$ denotes the set of polynomials over $x$ for some fixed
degree. Thus, while the problem of testing non-negativity of a polynomial is
intractable in general, by constraining the feasible set to SOS the problem
becomes tractable. The converse question of whether a non-negative polynomial
is necessarily a sum of squares is unfortunately false, indicating that this
test is conservative~\cite{Parrilo:2003fh}. Nonetheless, SOS feasibility is
sufficiently powerful for our purposes.

\subsection{The Positivstellensatz}

At this point it is possible to determine whether a particular polynomial,
possibly parameterized, is a sum of squares. The next step is to determine how
to combine multiple polynomial inequalities. The answer is given by the theorem
that has come to be known as Stengle's \emph{Positivstellensatz}.

\begin{theorem}[Stengle's Positivstellensatz~\cite{Stengle:1974ix}] 
The set
\begin{multline*}
	X=\big\{ x\mid f_{i}(x) \ge 0,\ h_{j}(x) = 0\\
		\text{ for all } i=1,\ldots,m, j=1,\ldots,p \big\}
\end{multline*}
is empty if and only if there exists $t_{i}\in\mathbb{R}[x]$,
and $s_{i},r_{ij},\ldots\in\Sigma[x]$ such that
\[
	-1=s_{0}+\sum_{i}h_{i}t_{i}+\sum_{i}s_{i}f_{i}
	+\sum_{i\neq j}r_{ij}f_{i}f_{j}+\cdots
\]
\end{theorem}
This powerful theorem allows for~\eqref{eq:sos_pde} to incorporate the domain
requirements $x\in\mathbb{S}$ and $x\in\partial\mathbb{S}$.

\section{Domain Decomposition}
\label{sec:domain_decomposition}
We first briefly review ADMM before demonstrating its use in domain
decomposition, following~\cite{Boyd:2011bw}.

\subsection{Alternating Direction Method of Multipliers}
\label{sec:admm}
The Alternating Direction Method of Multipliers (ADMM) will
serve as the basis for enforcing continuity and differentiability of $\Psi(x)$
on the boundaries of the decomposed regions. Other decomposition schemes are
possible, see~\cite{Bertsekas:1976,Bertsekas:1996} for a survey. ADMM is a
``meta''-optimization scheme, where each step is carried out by solving a
convex optimization problem. Consider the optimization
\begin{equation}\label{eq:admm_standard_form}
	\begin{array}{rl}
		\mbox{min.}   & f(x)+g(z) \\
		\mbox{s.t.} & Ax+Bz=c
	\end{array}
\end{equation}
over real vector variables $x$ and $z$, with convex functions $f$ and $g$.
Define an augmented Lagrangian
\[
	L_{\rho}=f(x)+g(z)+y^{T}\left(Ax+Bz-c\right)
		+\frac{\rho}{2}\left\|Ax+Bz-c\right\|_{2}^{2},
\]
where $\rho>0$ is an algorithm parameter, and $y$ is the dual variable
associated with the equality constraint. The constrained optimization is solved
through alternately minimizing the augmented Lagrangian over the primal
variables $x$, $z$, and updating the dual variable $y$,
\begin{eqnarray*}
x^{k+1} & := & \text{argmin}_{x}L_{\rho}(x,z^{k},y^{k})\\
z^{k+1} & := & \text{argmin}_{z}L_{\rho}(x^{k+1},z,y^{k})\\
y^{k+1} & := & y^{k}+\rho\left(Ax^{k+1}+Bz^{k+1}-c\right).
\end{eqnarray*}

The sum of squares formalism allows a general polynomial optimization problem
to be converted to a sequence of SDPs, where the variables are the polynomial
coefficients. ADMM extends readily to SDPs. To that end, consider
\[
	\begin{array}{rl}
		\mbox{min.}   & f(x) + g(z) \\
		\mbox{s.t.} 
			& Ax+Bz=c\\
			& x\in\mathcal{C}_1, \quad z\in\mathcal{C}_2,
	\end{array}
\]
where $x,z\in\mathbb{R}^n$ are the variables and $\mathcal{C}_1, \mathcal{C}_2$
are SDP-representable sets. With the same form $L_\rho$, the ADMM iterations
are quadratically penalized SDPs,
\begin{eqnarray*}
x^{k+1} & := & \text{argmin}_{x\in\mathcal{C}_1}L_{\rho}(x,z^{k},y^{k})\\
z^{k+1} & := & \text{argmin}_{z\in\mathcal{C}_2}L_{\rho}(x^{k+1},z,y^{k})\\
y^{k+1} & := & y^{k}+\rho\left(Ax^{k+1}+Bz^{k+1}-c\right).
\end{eqnarray*}
The only difference is the primal variables are now constrained to lie in the
spectrahedra (the convex set of semidefinite constraints~\cite{Boyd:2004uz})
$\mathcal{C}_1$ and $\mathcal{C}_2$.

The value in this decomposition is the attendant convergence guarantees
obtained with ADMM. In particular, we will make the following two assumptions,
which guarantee convergence:
\begin{assumption}
	\label{as:admm1}
	The (extended real valued) functions $f:\mathbb{R}^n \to \mathbb{R} \cup
	+\infty$ and $g : \mathbb{R}^m \to \mathbb{R} \cup +\infty$ are closed, proper,
	and convex.
\end{assumption}
\begin{assumption}
\label{as:admm2}
	The unaugmented Lagrangian has a saddle point.
\end{assumption}
If it can be demonstrated that the optimization problem obeys these
assumptions, then the following general theorem becomes available:
\begin{theorem}\label{thm:admm_convergence}
(See~\cite{Boyd:2011bw}) Given Assumptions \ref{as:admm1}, \ref{as:admm2} then
the ADMM iterates satisfy the following:
\begin{itemize}
	\item \textbf{Residual convergence}: $r^k \to 0$ as $k \to \infty$, i.e.,
		the iterates approach feasibility 
	\item \textbf{Objective convergence}: $f(x^k) + g(z^k) \to p^*$ as
		$k\to\infty$, i.e., the objective function of the iterates approaches the
		optimal value
	\item \textbf{Dual variable convergence}: $y^k \to y^*$ as $k \to \infty$,
		where $y^*$ is a dual optimal point
\end{itemize}
\end{theorem}

\subsection{Decomposition of Stochastic Optimal Control}
As the optimal control problem is assumed to take place over a compact state
space, the domain of \eqref{eq:HJB-linear} may decomposed into finitely many
regions $\mathcal{R}_j\subseteq\mathbb{R}^n$, $j=1,\ldots,N_R$. Assuming the
pairwise boundary between the regions may be described in terms of a
semialgebraic set, we have the following result,
\begin{theorem}\label{thm:posit_border}
Given desirability function $\Psi_{i}(x)$ valid on region $\mathcal{R}_i$,
$\Psi_{j}(x)$ valid on region $\mathcal{R}_j$, and shared boundary $\xi=\left\{
x\mid h(x) = 0\right\}$ between $\mathcal{R}_i$ and $\mathcal{R}_j$, we have
$\Psi_{i}(x)=\Psi_{j}(x)$ on $\xi$ if there exists $c(x)\in \mathbb{R}[x]$ such
that 
\[
	\Psi_{i}(x)-\Psi_{j}(x)+c(x)h(x)=0
\]
\end{theorem}
\begin{proof}
A straightforward result of the Positivstellensatz, see~\cite{Prajna:2003vr}
for details.
\end{proof}
Similarly, continuity of the $n$-th derivative may be easily incorporated as
well by imposing equality of the derivative along the boundary.

\subsection{Two Region Explicit Example}
In the following analysis, we demonstrate how this result can be used to bind
together optimization problems over a decomposed domain. To obtain a useful
policy, we will require the combined policy to be $C^1$ continuous.

For clarity, we examine a pair of bordering partitions $\mathcal{R}_1$ and
$\mathcal{R}_2$, with shared boundary $h(x)$. The polynomials are assumed to be
of bounded degrees, with $\Deg(\Psi_i(x))$ bounded by $d$ and $\Deg(c_i(x))$ by
$d-k$, for all $i,j$. In this case,
\begin{align*}
	\Psi_1(x) &= \alpha_0 + \alpha_1 x + \cdots + \alpha_d x^d\\
	\Psi_2(x) &= \beta_0 + \beta_1 x + \cdots + \beta_d x^d\\
	c_1(x) &= \theta_0 + \theta_1 x + \cdots + \theta_{d-k} x^{d-k}\\
	c_2(x) &= \mu_0 + \mu_1 x + \cdots + \mu_{d-k} x^{d-k},
\end{align*}
where $h(x) = \rho_0 + \rho_1 x + \cdots + \rho_k x^k$ defines the boundary
region. The continuity constraint 
\[
	\Psi_1(x) - \Psi_2(x) + c_{1}(x) h(x) = 0
\]
is equivalent to the coefficient matching constraints
\begin{align*}
	0 &= \alpha_0 - \beta_0 + (\theta_0\rho_0) \\
	0 &= \alpha_1 - \beta_1 + (\theta_0\rho_1 + \theta_1\rho_0) \\
	0 &= \alpha_2 - \beta_2 + (\theta_0\rho_2 + \theta_1\rho_1  + \theta_2\rho_0)\\
	  & \quad\quad\quad\vdots \\
	0 &= \alpha_d - \beta_d + (\theta_{d-k}\rho_k).
\end{align*}
Note that the coefficient matching constraints are affine in the decision
variables $\alpha_i$, $\beta_i$, $i=1,\ldots,d$, and $\theta_j$, $\mu_j$,
$j=1,\ldots,d-k$. The derivative constraint~\eqref{eq:ex1-d-cons} appends
additional coefficient matching constraints,
\begin{align*}
	0 &= \alpha_1 - \beta_1 + (\mu_0\rho_0) \\
	0 &= 2\alpha_2 - 2\beta_2 + (\mu_0\rho_1 + \mu_1\rho_0) \\
	0 &= 3\alpha_2 - 3\beta_2 + (\mu_0\rho_2 + \mu_1\rho_1  + \mu_2\rho_0)\\
	  & \quad\quad\quad\vdots \\
	0 &= d\alpha_d - d\beta_d + (\mu_{d-k}\rho_k).
\end{align*}
Continuity of higher order derivatives are incorporated similarly. The
continuity and derivative coefficient matching constraints, together with the
approximation error constraint~\eqref{eq:gam_min}, can be aggregated into
matrix form,
\[
	A^{(1)}z_1 + A^{(2)}z_2 = 0,
\]
where $z_1=(\alpha_0,\ldots,\theta_{d-k},\gamma_1)$ are the coefficients
associated with $\mathcal{R}_1$, and $z_2=(\beta_0,\ldots,\mu_{d-k},\gamma_2)$
are the coefficients associated with $\mathcal{R}_2$. It is now straightforward
to incorporate the affine matrix constraint into a dual decomposition scheme.
The decomposed variant of optimization \eqref{eq:sos_pde} is
\begin{flalign}
\mbox{min.}\,\, & \gamma_1 + \gamma_2 \\
\mbox{s.t.}\,\, & \frac{1}{\lambda} q \Psi_{1}\ge \partial_t \Psi_1+ L(\Psi_1),\quad x\in\mathcal{R}_1\\
 &  \frac{1}{\lambda} q \Psi_{2}\ge \partial_t \Psi_2+ L(\Psi_2) ,\quad x\in\mathcal{R}_2\\
 & \gamma_{1}-\left(\frac{1}{\lambda} q \Psi_{1}-rhs\right)\ge0,\quad x\in\mathcal{R}_1\\
 & \gamma_{2}-\left(\frac{1}{\lambda} q \Psi_{2}-rhs\right)\ge0,\quad x\in\mathcal{R}_2\\
 & \Psi_{1}(x)-\Psi_{2}(x)+c_{1}(x)x=0\label{eq:ex1-v-constraint}\\
 & \frac{\partial\Psi_{1}}{\partial x}(x)-\frac{\partial\Psi_{2}}{\partial x}(x)+c_{2}(x)x=0\label{eq:ex1-d-cons}\\
 & \gamma_1 = \gamma_2\label{eq:gam_min}
\end{flalign}
where the Positivstellensatz is used to enforce the domain restrictions
(see~\cite{Horowitz:2014tu} for details).  The coupling constraints
\eqref{eq:ex1-v-constraint} and \eqref{eq:ex1-d-cons} prevent decomposition
into two parallel optimizations. In addition, the objective is coupled through
the equality constraint \eqref{eq:gam_min}, which ensures that the maximum
pointwise approximation error over any region is no more than
$\gamma^\mathrm{max}=\gamma_1=\gamma_2$.

To wit, define the quadratically penalized Lagrangian
\begin{multline*}
	L_\rho(\gamma_1,z_1,\gamma_2,z_2,\lambda) 
	= \gamma_1 + \gamma_2 + \mathcal{I}_{\mathcal{C}_1}(z_1) + \mathcal{I}_{\mathcal{C}_2}(z_2) + \\
	+ \lambda^T(A^{(1)}z_1 + A^{(2)}z_2)
	+ \frac{\rho}{2} \left\|A^{(1)}z_1 + A^{(2)}z_2\right\|_2^2,
\end{multline*}
where $\mathcal{I}_{\mathcal{C}_i}(z_i)$ is the indicator function of the
optimization problem over each individual partition, obtained by reduction of
\eqref{eq:sos_pde} to semidefinite program form~\cite{Parrilo:2000ui}. The
alternating direction iteration may then be performed as
\begin{align}\label{eq:final_admm}
	(\gamma_1^{k+1}, z_1^{k+1}) &:= \arg\min_{\gamma_1,z_1} L_\rho(\gamma_1,z_1,\gamma_2^k,z_2^k,\lambda^k)\\
	(\gamma_2^{k+1}, z_2^{k+1}) &:= \arg\min_{\gamma_2,z_2} L_\rho(\gamma_1^{k+1},z_1^{k+1},\gamma_2,z_2,\lambda^k)\\
	\lambda^{k+1} &:= \lambda^k + \rho(A^{(1)}z_1^{k+1} + A^{(2)}z_2^{k+1}).
\end{align}

The above procedure may be repeated for all partitions $\mathcal{R}_i$ and
$\mathcal{R}_j$ that share a common boundary. Each minimization, a semidefinite
program, is taken over only those constraints associated with the specified
region. This achieves a degree of decoupling, limiting the size of the
polynomial optimization problem, and thus the semidefinite program, for each
individual partition.


\subsection{Parallelization}
A further decoupling may be achieved through a judicious choice of domain
partitions. This idea is well known in the partial differential equation
community~\cite{bjorstad2004domain}. Suppose partitions $\mathcal{R}_i$ and
$\mathcal{R}_j$ share no common border $h_{i,j}(x)$. As variables from disjoint
partitions are only shared through the common boundary constraints
\eqref{eq:ex1-v-constraint}, it is straightforward to see that $z_i^{k+1}$ and
$z_j^{k+1}$ are independent of one another. This allows for these optimizations
to be performed in parallel. One valid partition is to decompose the domain into
a checkerboard pattern, separating the domain into shaded and unshaded tiles.
As shaded tiles share no optimization variables with one another, they may be
optimized in parallel, and similar with the unshaded. By alternating between
shaded and unshaded, the correct descent direction continues to be taken,
guaranteeing convergence. See~\cite{Parikh:2014} for a detailed discussion of
parallelization ideas, and Fig.~\ref{fig:checkerboard} for an illustration of
this beneficial decomposition pattern.

\section{Analysis}
\label{sec:analysis}


%

A benefit of the sum of squares relaxation approach is that the solutions
produced are guaranteed to be upper and lower bounds (depending on the
direction of the inequality \eqref{eq:relax_hjb}) when performed over a single
partition~\cite{Horowitz:2014tu}. These guarantees are retained in the domain
decomposition setting.

\begin{theorem} \label{thm:Psi_bound}
Given a solution set $\left\{ \Psi_i,\gamma_i \right\} $ to the converged
optimization problem \eqref{eq:final_admm} where $C^2$ continuity is
enforced, and if $\Psi^*$ is the solution to \eqref{eq:HJB-linear}, then
$\Psi(x) \ge \Psi^{*}(x)$ for all $x\in\mathcal{R}_i$.
\end{theorem}
\begin{proof} \textit{(Sketch)} The derivation follows the proof of Theorem~5
in \cite{Horowitz:2014tu} with little modification. The only modification
arises from the fact that the elliptic and parabolic maximum principles rely on
$C^2$ continuity of the super-solution. As the solution is polynomial
on the interior of each boundary, and therefore infinitely differentiable, this
requirement needs only be enforced explicitly along the partition boundaries.
\end{proof}

A benefit of this approach is that not only may an upper bound be computed, but
in fact reversing the inequalities of the optimization results in an additional
optimization problem that can be used to find a pointwise lower bound to the
underlying optimal solution. As both upper and lower bounds are available, it
is possible to see the maximal possible error of the solution.
See~\cite{Horowitz:2014tu} for details.

\begin{figure}[t]
\centering
	\vspace{3mm}
    \begin{tikzpicture}[thick]
        \def \sqwid {1.33cm}
        \def \sqhig {1cm}
        \tikzstyle{shad}=[blue!30!white, draw=blue!50!white]
        \tikzstyle{unshad}=[white, draw=blue!50!white]
        \tikzstyle{rnode}=[anchor=north west, text=black!80!white]
        \foreach \i in {0,1} {
            \foreach \j in {0,1} {
                \filldraw[unshad]
                    (2*\j*\sqwid,2*\i*\sqhig)
                        rectangle
                    (2*\j*\sqwid+\sqwid, 2*\i*\sqhig+\sqhig);
                \filldraw[shad]
                    (2*\j*\sqwid+\sqwid, 2*\i*\sqhig)
                        rectangle
                    (2*\j*\sqwid+2*\sqwid,2*\i*\sqhig+\sqhig);
            }
            \foreach \j in {0,1} {
                \filldraw[shad]
                    (2*\j*\sqwid,2*\i*\sqhig+\sqhig)
                        rectangle
                    (2*\j*\sqwid+\sqwid, 2*\i*\sqhig+2*\sqhig);
                \filldraw[unshad]
                    (2*\j*\sqwid+\sqwid, 2*\i*\sqhig+\sqhig)
                        rectangle
                    (2*\j*\sqwid+2*\sqwid,2*\i*\sqhig+2*\sqhig);
            }
        }
        \node[rnode] at (0,4*\sqhig)
            {\tiny $\mathcal{R}_1$};
        \node[rnode] at (\sqwid,4*\sqhig)
            {\tiny $\mathcal{R}_2$};
        \node[rnode] at (2*\sqwid,4*\sqhig)
            {\tiny $\mathcal{R}_3$};
        \node[rnode] at (0,3*\sqhig)
            {\tiny $\mathcal{R}_5$};
        \node[rnode] at (\sqwid,3*\sqhig)
            {\tiny $\mathcal{R}_6$};
        \node[rnode] at (0*\sqwid,2*\sqhig)
            {\tiny $\mathcal{R}_9$};
        \node[rnode] at (3*\sqwid,\sqhig)
            {\tiny $\mathcal{R}_{16}$};
        \node at (2*\sqwid+0.5*\sqwid,4*\sqhig-0.5*\sqhig) {$\cdots$};
        \node at (0*\sqwid+0.5*\sqwid,2*\sqhig-0.5*\sqhig) {$\vdots$};
        \node at (2*\sqwid+0.5*\sqwid,2*\sqhig-0.5*\sqhig) {$\ddots$};
        \node at (0*\sqwid+0.5*\sqwid,4*\sqhig-0.5*\sqhig)
            {\scriptsize $\Psi_1(x)$};
        \node at (1*\sqwid+0.5*\sqwid,4*\sqhig-0.5*\sqhig)
            {\scriptsize $\Psi_2(x)$};
        \node at (0*\sqwid+0.5*\sqwid,3*\sqhig-0.5*\sqhig)
            {\scriptsize $\Psi_5(x)$};
        \node at (1*\sqwid+0.5*\sqwid,3*\sqhig-0.5*\sqhig)
            {\scriptsize $\Psi_6(x)$};
        \node at (3*\sqwid+0.5*\sqwid,1*\sqhig-0.5*\sqhig)
            {\scriptsize $\Psi_{16}(x)$};
        \draw[very thick] (0,3*\sqhig) -- (\sqwid,3*\sqhig);
        \draw[very thick] (\sqwid,3*\sqhig) -- (\sqwid,4*\sqhig);
        \coordinate(h15) at (0.5*\sqwid,3*\sqhig);
        \coordinate(h12) at (1*\sqwid,3.5*\sqhig);
        \draw[-latex]
            (2.1*\sqwid,2.5*\sqhig) node[anchor=west] {\scriptsize$h_{1,2}(x)$}
            to[out=180,in=-15] (h12.east);
        \draw[-latex]
            (1.1*\sqwid,1.5*\sqhig) node[anchor=west] {\scriptsize$h_{1,5}(x)$}
            to[out=120,in=-30] (h15.south);
    \end{tikzpicture}
    \caption{A particular grid domain decomposition with the partitions grouped
        into shaded and unshaded sets. As the sets of the same color require no
        consensus over their local variables, it is possible to perform the
        optimization over each set in parallel while maintaining the
        convergence properties of ADMM.}
    \label{fig:checkerboard}
\end{figure}
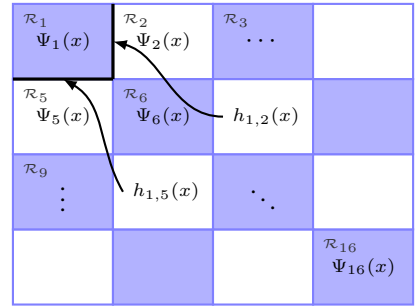

\section{Scalar Example}
\label{sec:illustrative_example}

We construct the optimization for a simple scalar example for illustrative
purposes. Consider the one dimensional system
\[
	dx=(x^{2}+u)\, dt + d\omega
\]
on the domain $x\in[-1,1]$. We have state cost $q(x)=1$, control cost
$R=1$, and parameter $\lambda=1$. We split the domain into regions
$\mathcal{R}_{1}=\left\{ x\mid x\in[-1,0]\right\} $, $\mathcal{R}_{2}=\left\{
x\mid x\in[0,1]\right\} $, creating $h_{1,2}(x)=x$.  For each of these problems
we form the optimization (\ref{eq:sos_pde}) on $\mathcal{R}_{1}$,
$\mathcal{R}_{2}$ independently.  To enforce equality of both the solution and
its derivative at the shared point $x=0$ we add the coupling constraints
\begin{align*}
\Psi_{1}(x)-\Psi_{2}(x)+c_{1}(x)x &= 0\\
\frac{\partial\Psi_{1}}{\partial x}(x)-\frac{\partial\Psi_{2}}{\partial x}(x)+c_{2}(x)x &= 0.
\end{align*}
To enforce the continuity constraint~\eqref{eq:ex1-v-constraint} for the point
boundary at the origin, it suffices to match the constant coefficients of
$\Psi_1$ and $\Psi_2$, i.e., we require $\Psi_1(0) = \Psi_2(0)$. This is an
affine constraint when the polynomial optimization is passed to an SDP. 


Numerical results for the one dimensional example are shown in
Fig.~\ref{fig:admm_iterations} and Fig.~\ref{fig:admm_gamma}. For simplicity,
the conditioning parameter was set to $\rho=1$, and the polynomial degree bound
to $6$ for each region.  Fig.~\ref{fig:admm_iterations} shows that within about
ten steps of ADMM, continuous differentiability at the boundary region $x=0$ is
achieved. Fig.~\ref{fig:admm_gamma} shows the evolution of the dual variables,
as well as the maximum approximation gap with iteration number. The SDP
optimization on each region was carried out on SDPT3 using YALMIP with the Sum
of Squares module~\cite{Lofberg:2009}.

\begin{figure}
	\psfrag{Iterations}[c][c]{Evolution of $\Psi_1(x)$ and $\Psi_2(x)$}
	\psfrag{x}[t][c]{\footnotesize $x$}
	\psfrag{Psi}[b][c]{\footnotesize $\Psi(x)$}
    \centering
    \includegraphics[width=0.75\columnwidth]{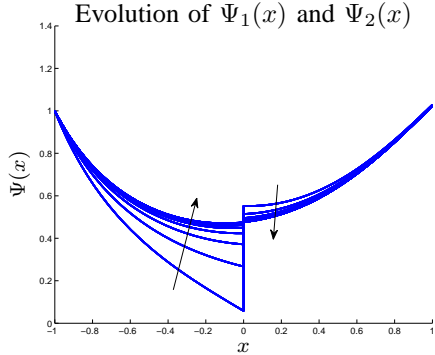}
	\caption{Evolution of the alternative value function over 10 ADMM steps.
	Arrows show direction of evolution.} 
	\label{fig:admm_iterations}
\end{figure}

\begin{figure}
	\psfrag{Dual variable value vs. step}[c][c]{}
	\psfrag{gammamax vs. step}[c][c]{}
	\psfrag{ADMM step}[t][c]{\footnotesize step}
	\psfrag{Price}[b][c]{\footnotesize dual}
	\psfrag{gammamax}[b][c]{\footnotesize $\gamma^\mathrm{max}$}
    \centering
    \includegraphics[width=0.49\columnwidth]{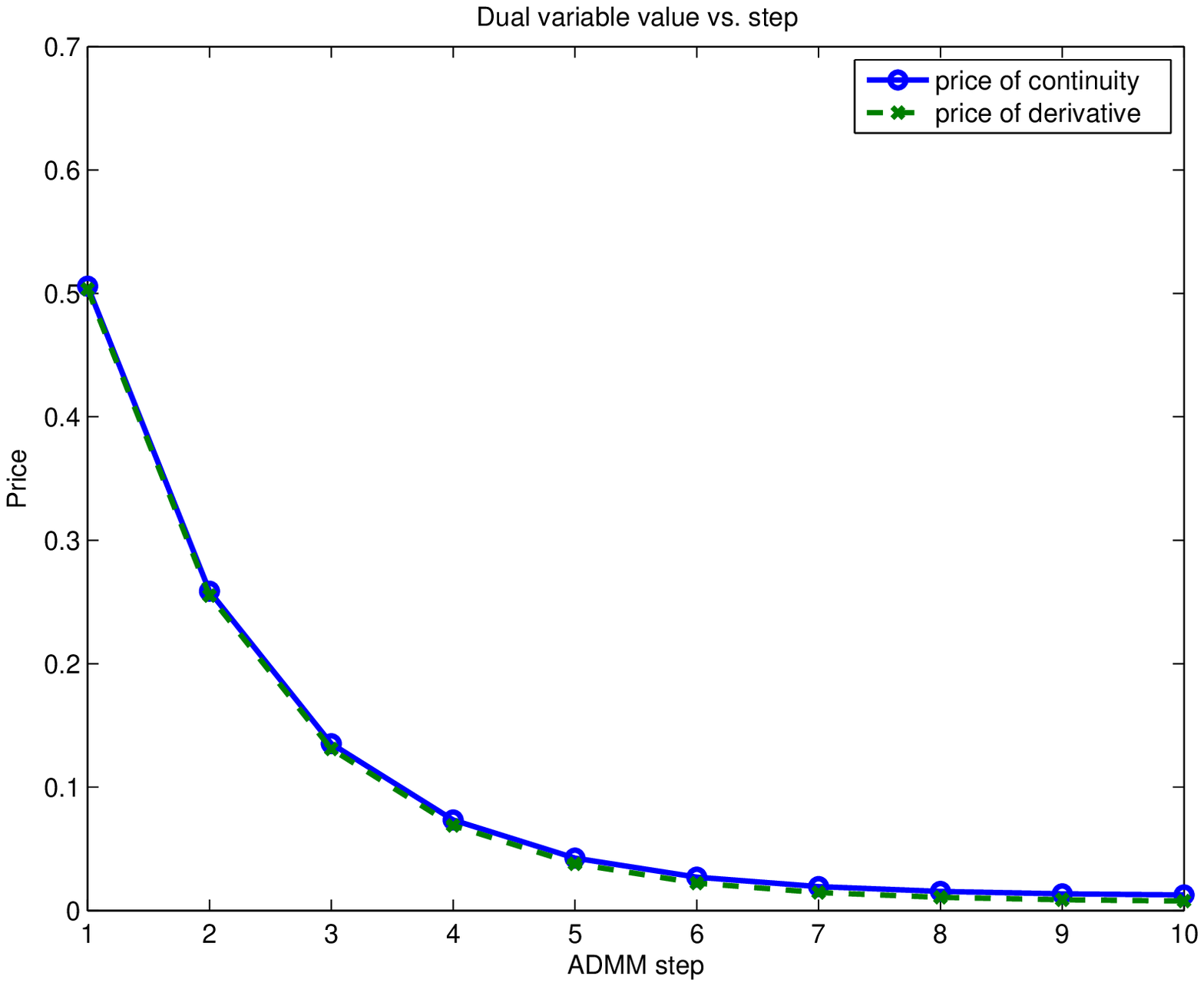}
    \includegraphics[width=0.49\columnwidth]{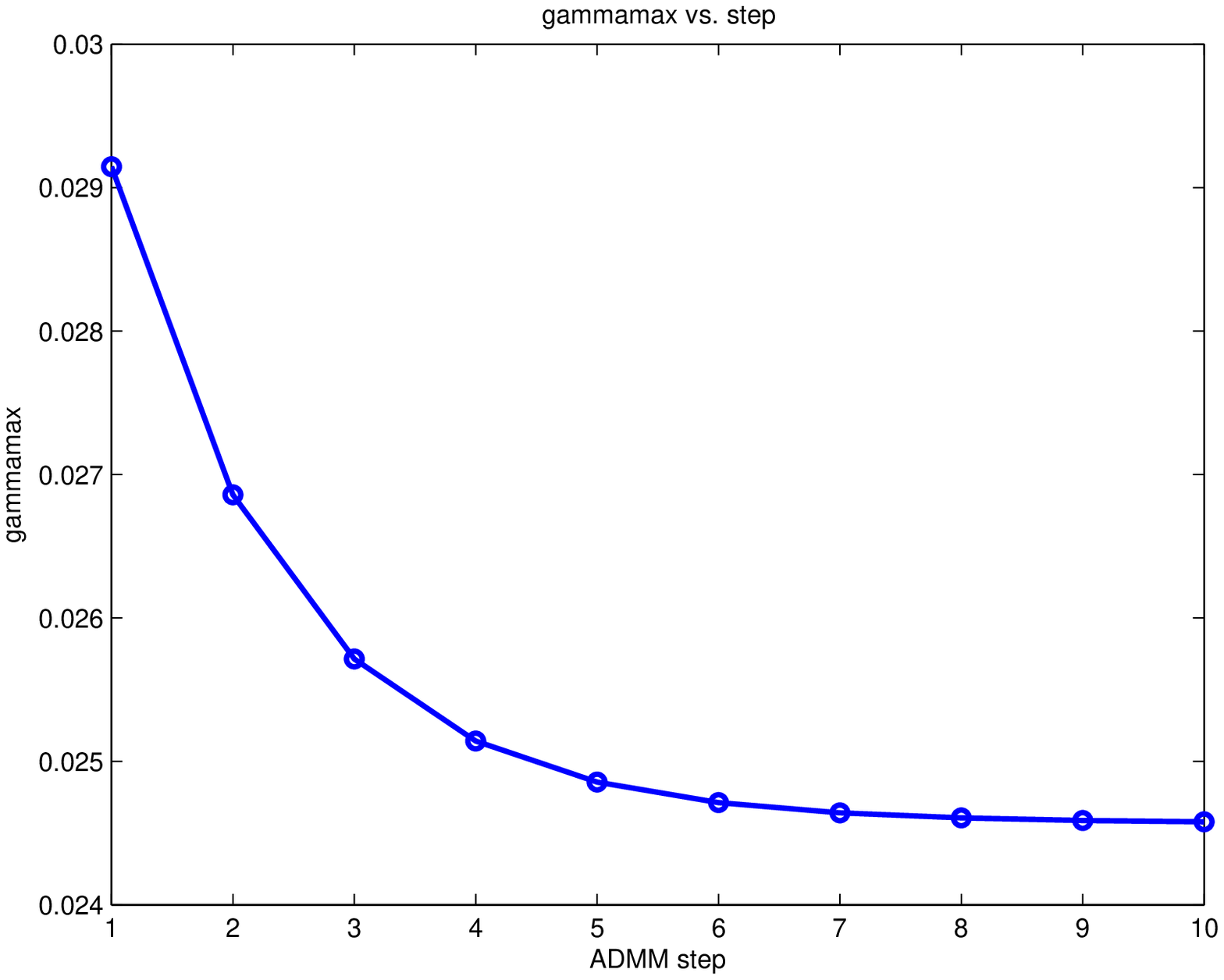}
	\caption{Values of the dual variables (left) and maximum approximation gap
	(right) with iteration number.} 
	\label{fig:admm_gamma}
\end{figure}

\section{Nonlinear Cartesian System}
\label{sec:nonlinear_robot}

\begin{figure*}[thbp]
	\psfrag{x}[c][c]{\footnotesize $x$}
	\psfrag{y}[c][c]{\footnotesize $y$}
    \addtolength{\subfigcapskip}{-0.4cm}
    \centerline{
    \subfigure[$C^0$-continuous approximation]{
		\psfrag{Psi}[t][b]{\small $\Psi(x,y)$, $d=8$, $n_r=3$, $\gamma^\mathrm{max}=0.57938$}
        \includegraphics[width=0.7\columnwidth]{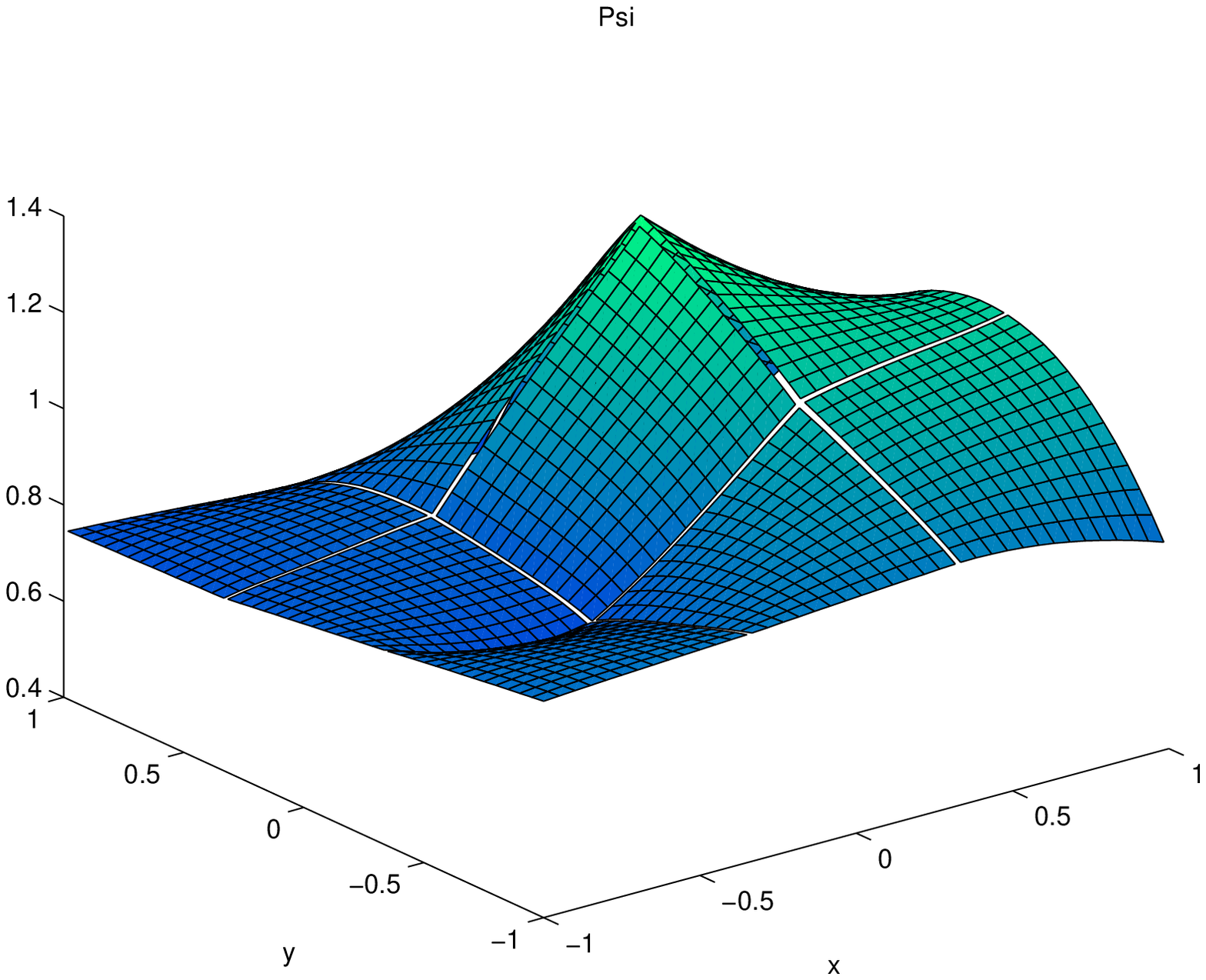}
        \label{fig:d8nr3}
    }
	\hspace{-5mm}
	\subfigure[$C^1$-continuous approximation]{
		\psfrag{Psi}[t][b]{\small $\Psi(x,y)$, $d=8$, $n_r=3$, $\gamma^\mathrm{max}=0.3123$}
        \includegraphics[width=0.7\columnwidth]{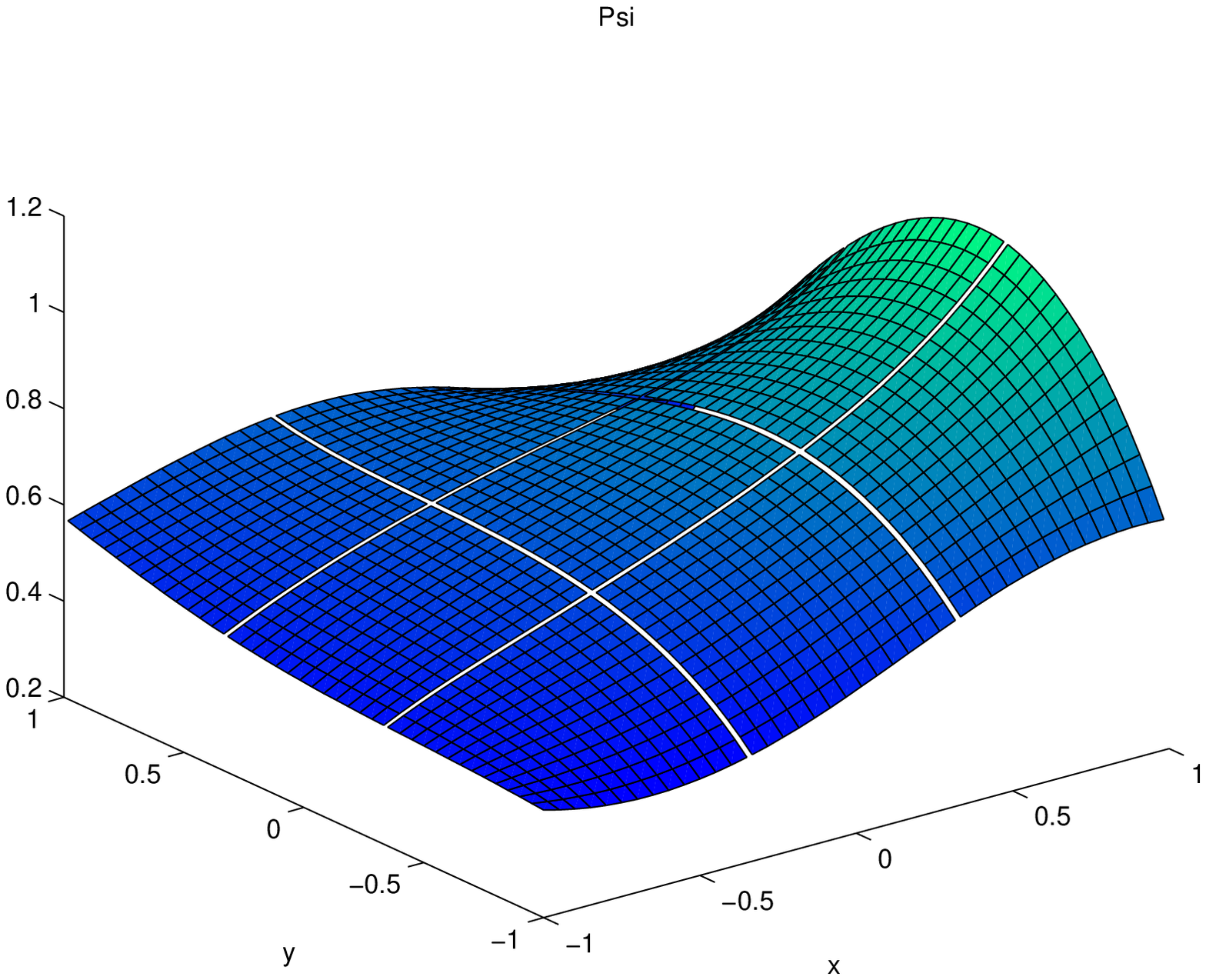}
        \label{fig:d8nr3_withderiv}
	}
	\hspace{-5mm}
    \subfigure[high fidelity approximation]{
		\psfrag{Psi}[t][b]{\small $\Psi(x,y)$, $d=14$, $n_r=1$, $\gamma^\mathrm{max}=0.0978$}
        \includegraphics[width=0.7\columnwidth]{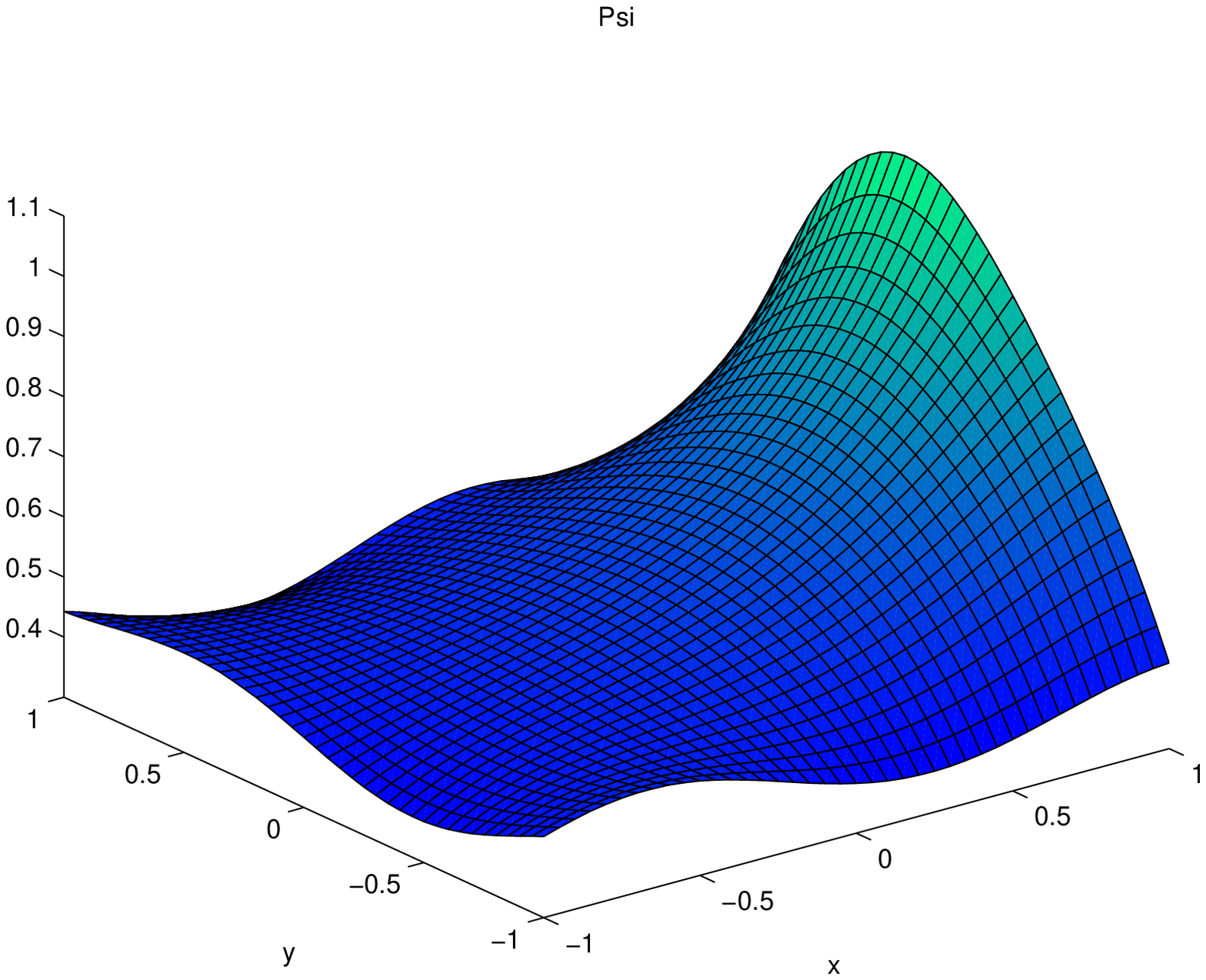}
        \label{fig:d14nr1}
    }
    }
    \caption{Results of multidimensional, nonlinear example.}
	\label{fig:Results-of-multidimensional}
\end{figure*}

To demonstrate the versatility of the method, a nonlinear, multidimensional
problem was solved with the following dynamics,
\[
	\begin{bmatrix}
		dx\\
		dy
	\end{bmatrix}
	=
	\left(0.1
	\begin{bmatrix}
		-2x-x^{3}-5y-y^{3}\\
		6x+x^{3}-3y-y^{3}
	\end{bmatrix}
	+
	\begin{bmatrix}
		u_1\\
		u_2
	\end{bmatrix}
	\right)
	dt
	+
	\begin{bmatrix}
		d\omega_1\\
		d\omega_2
	\end{bmatrix}.
\]

The problem is framed as a first exit problem, with the three sides of a square
domain $\mathbb{S}=[-1,1^2]$ given a unit penalty $\phi(x,y)=1$, while on the
remaining edge at $x=1$ a reward was given for achieving the center of the edge
with $\phi(x,y)=1-(y-1)^2$. Representative alternative value function
approximations are shown in Fig.~\ref{fig:Results-of-multidimensional}. In
Table~\ref{tab:gamma-results} we also summarize the maximum approximation gap
$\gamma^\mathrm{max}$ for a checkerboard decomposition of $\mathbb{S}$ with
$n_r$ regions per dimension, and approximating polynomial degree bound $d$ in
each region.

\begin{table}[!t]
\caption{Slack value $\gamma^\mathrm{max}$ as a function of polynomial degree
$d$, and number of regions $n_r$ per dimension.}
\label{tab:gamma-results}
\begin{tabular}{r|*{6}l}
\hline
 & \multicolumn{6}{c}{$d$}\\
$n_r$ & 4		& 6			& 8			& 10		& 12		& 14\\
\hline
1 & 6.8374	& 2.5085	& 0.6344	& 0.3501	& 0.0804	& 0.0978\\
2 & 6.7065	& 2.1561	& 0.6399	& 0.3642	& 0.0859	&\\
3 &	6.4688	& 2.0579	& 0.5794	& 0.3304	&			&\\
4 &	6.2662	& 2.0689	& 0.5591	& 0.3005	&			&\\
5 &	6.6289	& 1.8812	& 0.5919	& 0.2917	&			&\\
6 &	6.3017	& 1.7638	& 0.5716	&			&			&\\
7 &	6.3178	& 1.6533	& 0.5403	& 			&			&\\
\hline
\end{tabular}
\end{table}



\section{Conclusion}
\label{sec:discussion}

A method to perform domain decomposition on stochastic optimal control problems
has been developed, allowing for local polynomial approximations to the
Hamilton Jacobi Bellman equation to be generated in parallel. Of importance is
the fact that the sum of squares relaxation used does not fundamentally rely on
the particular structure of the HJB PDE. In fact,~\cite{Horowitz:2014tu}
demonstrates that the technique may be readily applied to any linear parabolic
or elliptic PDE to obtain guaranteed upper and lower bounds over the domain.
The domain splitting of this work extends as well, allowing for local upper and
lower bounds to any linear PDE to be generated via optimization. While more
involved than existing numerical techniques such as the Finite Element method,
these techniques have formal guarantees that do not require an asymptotic limit
in discretization mesh size.

\addtolength{\textheight}{-7.1cm}   

A more direct implication lies in the generation of stabilizing controllers for
nonlinear systems. Until now, there has not existed a method to generate
near-optimal Control Lyapunov Functions for arbitrary nonlinear, stochastic
systems~\cite{Sontag:1983cq}. These domain decomposition techniques improve the
ability for optimal control policies to respond to system dynamics, enlarging
the class of systems that can be handled. Furthermore, existing results on sum
of squares in Lyapunov functions can be used to verify the stability of any
policy produced by these decomposition methods.

\subsection{Future Directions}

It is straightforward to recognize that many domain decompositions, such as the
checkerboard pattern illustrated, produce highly structured sparsity patterns
in the semidefinite program's constraint matrices. Such sparsity structures
have previously been used to significantly improve the computational cost of
large scale semidefinite and sum of squares
programs~\cite{henrion2009gloptipoly, waki2008algorithm}, work that could
easily be applied here as well. It is also an interesting question as to how
sparse basis functions~\cite{Horowitz:hd} might be incorporated into the
domain decomposition approach.

\bibliographystyle{IEEEtran}
\bibliography{references}

\end{document}